\documentclass[a4paper,12pt,reqno]{amsart}

\usepackage{a4wide, amssymb, amsmath, amsthm, graphics, comment, xspace, enumerate}
\usepackage{graphicx}
\usepackage{algorithmic, algorithm}

\newcommand{\NN}{\mathbb N}

\newcommand{\CC}{\mathbb C}
\newcommand{\RR}{\mathbb R}
\newcommand{\ZZ}{\mathbb Z}

\newcommand{\EE}{\mathcal E}
\newcommand{\DD}{\mathcal D}
\newcommand{\SSS}{\mathcal S}

\theoremstyle{plain}
\newtheorem{theorem}{Theorem}[section]
\newtheorem{proposition}[theorem]{Proposition}
\newtheorem{lemma}[theorem]{Lemma}
\newtheorem{corollary}[theorem]{Corollary}

\theoremstyle{remark}
\newtheorem{remark}[theorem]{Remark}

\theoremstyle{definition}

\allowdisplaybreaks

\numberwithin{equation}{section}

\newcommand{\supp}{\operatorname{supp}}

\begin{document}

\title[Extension of Localisation Operators]{Extension of Localisation Operators to Ultradistributional Symbols With Super-Exponential Growth}

\author[S. Pilipovi\'c]{Stevan Pilipovi\'c}
\address{Stevan Pilipovi\'c, Department of Mathematics and Informatics,
University of Novi Sad, Trg Dositeja Obradovi\'{c}a 4, 21000 Novi Sad, Serbia}
\email{stevan.pilipovic@dmi.uns.ac.rs}

\author[B. Prangoski]{Bojan Prangoski}
\thanks{B. Prangoski was partially supported by the bilateral project ``Microlocal analysis and applications'' funded by the Macedonian and Serbian academies of sciences and arts}
\address{Bojan Prangoski, Faculty of Mechanical Engineering\\ Ss. Cyril and Methodius University in Skopje \\ Karpos II bb \\ 1000 Skopje \\ Macedonia}
\email{bprangoski@yahoo.com}

\author[\DJ. Vuckovi\' c]{\DJ or\dj  e Vu\v{c}kovi\' c}
\address{\DJ or\dj e Vu\v{c}kovi\' c, Mathematical Institute of the Serbian Academy of Sciences and Arts, Belgrade, Serbia}
\email{djvuckov@turing.mi.sanu.ac.rs}
\thanks{\DJ. Vu\v{c}kovi\' c was partially supported by the MI SANU}


\date{}



\begin{abstract}
In the Gelfand-Shilov setting, the localisation operator $A^{\varphi_1,\varphi_2}_a$ is equal to the Weyl operator whose symbol is the convolution of $a$ with the Wigner transform of the windows $\varphi_2$ and $\varphi_1$. We employ this fact, to extend the definition of localisation operators to symbols $a$ having very fast super-exponential growth by allowing them to be mappings from $\DD^{\{M_p\}}(\RR^d)$ into $\DD'^{\{M_p\}}(\RR^d)$, where $M_p$, $p\in\NN$, is a non-quasi-analytic Gevrey type sequence. By choosing the windows $\varphi_1$ and $\varphi_2$ appropriately, our main results show that one can consider symbols with growth in position space of the form $\exp(\exp(l|\cdot|^q))$, $l,q>0$.
\end{abstract}

\maketitle

\section{Introduction}

The usual framework for the Anti-Wick calculus, and localisation operators in general, are the Schwartz distributions $\mathcal S'(\RR^d)$ and their natural extensions the Gelfand-Shilov spaces $\mathcal S'^{\{M_p\}}(\RR^d)$, where $M_p$, $p\in\NN$, is a Gevrey type sequence which controls the regularity and decay of the test functions. The localisation operator $A_a^{\varphi_1,\varphi_2}$ with symbol $a\in\SSS'^{\{M_p\}}(\RR^{2d})$ and windows $\varphi_1,\varphi_2\in\SSS^{\{M_p\}}(\RR^d)$ is usually defined via the short-time Fourier transform $V_{\varphi_1}$ and the adjoint $V^*_{\varphi_2}$ to $V_{\varphi_2}$:
\begin{equation}\label{locope-for-nes-b-i-neesk}
A^{\varphi_1,\varphi_2}_a:\SSS^{\{M_p\}}(\RR^d)\rightarrow \SSS'^{\{M_p\}}(\RR^d),\, A^{\varphi_1,\varphi_2}_a \psi= V^*_{\varphi_2}(aV_{\varphi_1}\psi)
\end{equation}
(or with $\SSS$ and $\SSS'$ in place of $\SSS^{\{M_p\}}$ and $\SSS'^{\{M_p\}}$ in the distributional setting); when both $\varphi_1$ and $\varphi_2$ are the Gaussian, \eqref{locope-for-nes-b-i-neesk} is the Anti-Wick quantisation of $a$. Their properties (as a quantisation procedure) and applications in the theory of pseudo-differential operators are extensively studied (see for example \cite{berezin,engl,T2,luef-skr,NR,PP,ppv-j-ana-math-phys,Shubin,teofanov}); they are also extensively investigated as an important tool in time-frequency analysis (see \cite{cord-gro,dou,teofanov-d,T1,wong}). The key property of the localisation operators is their close connection with the Weyl quantisation:
\begin{equation}\label{equ-rel-loc-weyl-quat-eqt1}
A^{\varphi_1,\varphi_2}_a=b^w,\quad \mbox{with}\quad b:= a*W(\varphi_2,\varphi_1),
\end{equation}
where $W(\varphi_2,\varphi_1)$ is the Wigner transform of $\varphi_2$ and $\varphi_1$ and $b^w$ stands for the Weyl quantisation of the (ultra)distribution $b$. The aim of this article is to extend the definition of localisation operators to symbols having very fast growth by considering them as mappings from $\DD^{\{M_p\}}(\RR^d)$ to $\DD'^{\{M_p\}}(\RR^d)$ ($\DD'^{\{M_p\}}(\RR^d)$ is the Gevrey class of ultradistributions associated to the non-quasi-analytic sequence $M_p$, $p\in\NN$, and $\DD^{\{M_p\}}(\RR^d)$ is the corresponding test space of compactly supported functions). The general idea is to employ \eqref{equ-rel-loc-weyl-quat-eqt1} as a definition by allowing the convolution to be considered in $\DD'^{\{M_p\}}$-sense (see Subsection \ref{ultradis-con-def}) provided the Weyl quantisation $b^w$ is a well-defined and continuous mapping from $\DD^{\{M_p\}}(\RR^d)$ into $\DD'^{\{M_p\}}(\RR^d)$. For this purpose, in Section \ref{wigner-tra-ult-dif-rdc}, we analyse the Wigner transform of ultradifferentiable functions with super-exponential decay. The main results of the article, Theorem \ref{ext22} and Theorem \ref{ext221}, give classes of windows and symbols for which this procedure can be implemented. For example, they are applicable with:
\begin{itemize}
\item[$(i)$] $\varphi_1(x)=\varphi_2(x)=e^{-r(1+|x|^2)^{q/2}}$, $q\geq 1$, $r>0$, and $a$ a finite sum of terms of the form $(P(D)(e^{l(1+|x|^2)^{q/2}}))\otimes g(\xi)$ where $l\in(0,2r)$, $g\in\SSS'^{\{M_p\}}(\RR^d)$ and $P(D)$ is any (ultra)differential operator; notice that $a$ can grow exponentially (almost) twice as fast in the $x$-variable compared to the decay of $\varphi_1$ and $\varphi_2$;
\item[$(ii)$] $\varphi_1(x)=\varphi_2(x)=\exp(-e^{(1+|x|^2)^q})$, $q>0$, and $a$ a finite sum of terms of the form $(P(D)(\exp(e^{l(1+|x|^2)^q})))\otimes g(\xi)$ where $l\in(0,1)$, $g\in\SSS'^{\{M_p\}}(\RR^d)$ and $P(D)$ is any (ultra)differential operator.
\end{itemize}
It is important to point out that a similar general idea is employed in \cite[Section 5]{PP} for the Anti-Wick operators, i.e. the case when $\varphi_1=\varphi_2=$ the Gaussian; in this case, the problem boils down to studying the convolution of ultradistributions with a Gaussian function (recall that the Wigner transform of a Gaussian is again a Gaussian).\\
\indent Finally, we mention that, by employing the mapping properties of the short-time Fourier transform from \cite{and-jasson,DV}, it is possible to extend the definition of the localisation operators with symbols belonging to tensor products of anisotropic Gelfand-Shilov spaces (see Remark \ref{rem-for-res-ajr} below). The benefit of our approach is that one can be more precise on the allowed growth of the symbols and windows.

\section{Preliminaries}

We employ the standard notations: $\langle x\rangle=(1+|x|^2)^{1/2}$, $x\in\RR^d$, and $D^{\alpha}=D^{\alpha_1}_{x_1}\ldots D^{\alpha_d}_{x_d}$, $\alpha\in\NN^d$, where $D^{\alpha_j}_{x_j}=i^{-|\alpha_j|}\partial^{\alpha_j}/\partial x_j^{\alpha_j}$, $j=1,\ldots,d$. We fix the constants in the Fourier transform as follows: $\mathcal{F}f(\xi)=\int_{\RR^d} e^{-2\pi i x\xi}f(x)dx$, $f\in L^1(\RR^d)$.\\
\indent A sequence of positive numbers $(M_p)_{p\in\NN}$ will be called a \textit{weight sequence} if $M_0=1$, $M_p^{1/p}\rightarrow\infty$, as $p\rightarrow \infty$, and it satisfies the following conditions (\cite{Komatsu1}):\\
\indent $(M.1)$ $M_{p}^{2} \leq M_{p-1} M_{p+1}$, $p \in\ZZ_+$;\\
\indent $(M.2)$ $M_{p} \leq c_0H^{p} \min_{0\leq q\leq p} \{M_{p-q} M_{q}\}$, $p,q\in \NN$, for some $c_0,H\geq1$.\\
\noindent Sometimes we will impose the following additional condition on a weight sequence $(M_p)_{p\in\NN}$\\
\indent $(M.3)'$ $\sum_{p=1}^{\infty}M_{p-1}/M_p<\infty$,\\
\noindent or the stronger condition\\
\indent $(M.3)$ $\sum_{j=p+1}^{\infty}M_{j-1}/M_j\leq c_0pM_p/M_{p+1}$, $p\in\ZZ_+$, for some $c_0\geq 1$;\\
we will always emphasise when we impose either one of these conditions. The Gevrey sequence $p!^{\sigma}$, $p\in\NN$, satisfies $(M.1)$ and $(M.2)$ when $\sigma>0$; when $\sigma>1$ it also satisfies $(M.3)$.\\
\indent Given a weight sequence $(M_p)_{p\in\NN}$, we set $m_p=M_p/M_{p-1}$, $p\in\ZZ_+$. For $\alpha\in\NN^d$, we use the notation $M_{\alpha}$ for $M_{|\alpha|}$, $|\alpha|=\alpha_1+...+\alpha_d$. If $(A_p)_{p\in\NN}$ is another weight sequences, the notation $A_p\subset M_p$ means $A_p\leq CL^pM_p$, $p\in\NN$, for some $C,L>0$. Following Komatsu \cite{Komatsu1}, we define the associated function for the weight sequence $(M_{p})_{p\in\NN}$ by
\begin{equation}\label{assoc-fun-sec}
M(\rho):=\sup_{p\in\NN}\ln   \rho^{p}/M_{p},\quad \rho> 0.
\end{equation}
It is a non-negative, continuous, monotonically increasing function, which vanishes for sufficiently small $\rho>0$ and increases more rapidly then $\ln \rho^p$, when $\rho$ tends to infinity, for any $p\in\NN$. When $M_p=p!^{\sigma}$, $\sigma>0$, we have $M(\rho)\asymp \rho^{1/\sigma}$. Employing the inequality $(\lambda+\rho)^p\leq 2^p\lambda^p+2^p\rho^p$, $\lambda,\rho\geq 0$, $p\in\NN$, one easily verifies the bound
\begin{equation}\label{ine-for-seq-assoc}
e^{M(\lambda+\rho)}\leq 2e^{M(2\lambda)}e^{M(2\rho)},\quad \mbox{for all}\,\,\lambda,\rho\geq 0.
\end{equation}
Notice that the condition $(M.2)$ is not needed for the definition of the associated function \eqref{assoc-fun-sec} nor for the proof of \eqref{ine-for-seq-assoc}.\\
\indent Assume that the weight sequence $(M_p)_{p\in\NN}$ satisfies $(M.3)'$. Let $U\subseteq\RR^d$ be an open set and $K$ a regular compact subset of $U$ (i.e. $\overline{\operatorname{int}K}=K$). Then $\EE^{M_p,h}(K)$, $h>0$, is the space of all $\varphi\in \mathcal{C}^{\infty}(U)$ which satisfy
\begin{equation}\label{n-ult-dif-fun-whl}
\sup_{\alpha\in\NN^d}\sup_{x\in K} h^{|\alpha|}|D^{\alpha}\varphi(x)|/M_{\alpha}<\infty,
\end{equation}
and $\DD^{M_p,h}_K$, $h>0$, is the space of all $\varphi\in \mathcal{C}^{\infty}(\RR^d)$ with support in $K$ that satisfy \eqref{n-ult-dif-fun-whl}. Following Komatsu \cite{Komatsu1}, we define the locally convex spaces
$$
\EE^{\{M_p\}}(U)=\lim_{\substack{\longleftarrow\\ K\subset\subset U}}
\lim_{\substack{\longrightarrow\\ h\rightarrow 0^+}} \EE^{M_p,h}(K),\quad \DD^{\{M_p\}}(U)=\lim_{\substack{\longrightarrow\\ K\subset\subset U}}\lim_{\substack{\longrightarrow\\ h\rightarrow 0^+}} \DD^{M_p,h}_K.
$$
The spaces of ultradistributions and ultradistributions with compact support of Roumieu type are defined as the strong duals of $\DD^{\{M_p\}}(U)$ and $\EE^{\{M_p\}}(U)$ respectively; we refer to \cite{Komatsu1,Komatsu2,Komatsu3} for the properties of these spaces.\\
\indent Let $(M_p)_{p\in\NN}$ and $(A_p)_{p\in\NN}$ be two weight sequences which satisfy the following: there exists $\sigma\in(0,1)$ such that $p!^{\sigma}\subset M_p$ and $p!^{1-\sigma}\subset A_p$. For $h>0$, we denote by $\SSS^{M_p,h}_{A_p,h}(\RR^d)$ the Banach space of all $\varphi\in\mathcal{C}^{\infty}(\RR^d)$ which satisfy
$$
\sup_{\alpha\in \NN^d}h^{|\alpha|}\|e^{A(h|\cdot|)} D^{\alpha}\varphi\|_{L^{\infty}(\RR^d)} /M_{\alpha}<\infty.
$$
The space of tempered ultradifferentiable function of Roumieu type is defined by
$$
\SSS^{\{M_{p}\}}_{\{A_p\}}(\RR^d)=\lim_{\substack{\longrightarrow\\ h\rightarrow 0^+}} \SSS^{M_{p},h}_{A_p,h}(\RR^d).
$$
Its strong dual $\SSS'^{\{M_{p}\}}_{\{A_p\}}(\RR^d)$ is the spaces of tempered ultradistributions of Roumieu type. We refer to \cite{PilipovicK,and-jasson-lenny,and-jasson,PPV-JMPA} for their topological properties. When $M_p=A_p$, we will simply write $\SSS^{\{M_p\}}(\RR^d)$ and $\SSS'^{\{M_p\}}(\RR^d)$ instead of $\SSS^{\{M_p\}}_{\{M_p\}}(\RR^d)$ and $\SSS'^{\{M_p\}}_{\{M_p\}}(\RR^d)$ respectively. In this case, the Fourier transform is a topological isomorphism on $\SSS^{\{M_p\}}(\RR^d)$ and on $\SSS'^{\{M_p\}}(\RR^d)$.\\
\indent An entire function $P:\CC^d\rightarrow \CC$, $P(z)=\sum_{\alpha}c_{\alpha}z^{\alpha}$, is called an \textit{ultrapolynomial of class} $\{M_p\}$ if for every $L>0$ there is $C>0$ such that $|c_{\alpha}|\leq CL^{|\alpha|}/M_{\alpha}$, for all $\alpha\in\NN^d$. The corresponding operator $P(D)=\sum_{\alpha} c_{\alpha}D^{\alpha}$ is called an \textit{ultradifferential operator of class} $\{M_p\}$ and it acts continuously on $\SSS^{\{M_p\}}_{\{A_p\}}(\RR^d)$ and $\SSS'^{\{M_p\}}_{\{A_p\}}(\RR^d)$ as well as on $\EE^{\{M_p\}}(U)$, $\DD^{\{M_p\}}(U)$, $\EE'^{\{M_p\}}(U)$ and $\DD'^{\{M_p\}}(U)$.\\
\indent We denote by $\mathfrak{R}$ the set of all positive monotonically increasing sequences $(r_p)_{p\in\ZZ_+}$ such that $r_p\rightarrow \infty$, as $p\rightarrow \infty$. With the partial order relation $(r_p)\leq (k_p)$ if $r_p\leq k_p$, $\forall p\in\ZZ_+$, $(\mathfrak{R},\leq)$ becomes a directed set (both upwards and downwards directed). For $(r_p)\in\mathfrak{R}$, we denote by $N_{r_p}(\cdot)$ the associated function to the sequence $M_p\prod_{j=1}^pr_j$, $p\in\NN$;\footnote{Here and throughout the rest of the article we employ the principle of vacuous (empty) product, i.e. $\prod_{j=1}^0 r_j=1$.} although $M_p\prod_{j=1}^pr_j$, $p\in\NN$, may fail to satisfy $(M.2)$ we can still define its associated function as in \eqref{assoc-fun-sec} and \eqref{ine-for-seq-assoc} holds true for $e^{N_{r_p}(\cdot)}$ as well. Notice that for any $(r_p)\in\mathfrak{R}$ and $h>0$, $e^{N_{r_p}(\lambda)}e^{-M(h\lambda)}\rightarrow 0$, as $\lambda\rightarrow \infty$.\\
\indent A measurable function $f$ on $\RR^d$ is said to have \textit{ultrapolynomial growth of class} $\{M_p\}$ if $e^{-M(h|\cdot|)}f\in L^{\infty}(\RR^d)$, for all $h>0$. We will need the following fact \cite[Lemma 2.1]{PP3}: $f\in\mathcal{C}(\RR^d)$ is of ultrapolynomial growth of class $\{M_p\}$ if and only if
\begin{equation}\label{equ-for-con-equ-for-gro}
\exists (r_p)\in\mathfrak{R}\,\, \mbox{such that}\,\, e^{-N_{r_p}(|\cdot|)}f\in L^{\infty}(\RR^d).
\end{equation}

\subsection{$\DD'^{\{M_p\}}$-convolution of ultradistributions}\label{ultradis-con-def}

Let $(M_p)_{p\in\NN}$ be a weight sequence that satisfies $(M.3)$. For $h>0$, we denote by $\DD^{M_p,h}_{L^{\infty}}(\RR^d)$ the space of all $\varphi\in\mathcal{C}^{\infty}(\RR^d)$ which satisfy $\sup_{\alpha\in\NN^d} h^{|\alpha|}\|D^{\alpha}\varphi\|_{L^{\infty}(\RR^d)}/M_{\alpha}<\infty$. The locally convex space $\DD^{\{M_p\}}_{L^{\infty}}(\RR^d)$ is defined as follows
$$
\DD^{\{M_p\}}_{L^{\infty}}(\RR^d)=\lim_{\substack{\longrightarrow\\ h\rightarrow 0^+}} \DD^{M_p,h}_{L^{\infty}};
$$
it is a complete barrelled and bornological space (see \cite[Section 7]{D-P-P-V}). We denote by $\dot{\mathcal{B}}^{\{M_p\}}(\RR^d)$ the closure of $\DD^{\{M_p\}}(\RR^d)$ in $\DD^{\{M_p\}}_{L^{\infty}}(\RR^d)$ (or equivalently, the closure of $\SSS^{\{M_p\}}(\RR^d)$ in $\DD^{\{M_p\}}_{L^{\infty}}(\RR^d)$); $\dot{\mathcal{B}}^{\{M_p\}}(\RR^d)$ is again complete barrelled and bornological space (see \cite[Section 7]{D-P-P-V}). The strong dual of $\dot{\mathcal{B}}^{\{M_p\}}(\RR^d)$ is denoted by $\DD'^{\{M_p\}}_{L^1}(\RR^d)$ and the strong dual of the latter (i.e. the strong bidual of $\dot{\mathcal{B}}^{\{M_p\}}(\RR^d)$) is exactly $\DD^{\{M_p\}}_{L^{\infty}}(\RR^d)$; see \cite[Theorem 7.3]{D-P-P-V}. We introduce a weaker locally convex topology on $\DD^{\{M_p\}}_{L^{\infty}}(\RR^d)$ with the following family of seminorms:
$$
\sup_{\alpha\in\NN^d}\frac{\|gD^{\alpha}\varphi\|_{L^{\infty}(\RR^d)}} {M_{\alpha}\prod_{j=1}^{|\alpha|}r_j},\quad (r_p)\in\mathfrak{R},\,\, g\in\mathcal{C}_0(\RR^d)\,\, \mbox{satisfying}\,\, g(x)>0,\, \forall x\in\RR^d,
$$
and we denote by $\DD^{\{M_p\}}_{L^{\infty},c}(\RR^d)$ the space $\DD^{\{M_p\}}_{L^{\infty}}(\RR^d)$ equipped with this topology. Then, $\DD^{\{M_p\}}_{L^{\infty},c}(\RR^d)$ is a complete space, $\DD^{\{M_p\}}(\RR^d)$ is dense in it and the strong dual of $\DD^{\{M_p\}}_{L^{\infty},c}(\RR^d)$ is topologically isomorphic to $\DD'^{\{M_p\}}_{L^1}(\RR^d)$ (see \cite[Section 5]{PPV-JMPA}).\\
\indent Given $f,g\in\DD'^{\{M_p\}}(\RR^d)$, we say that the $\DD'^{\{M_p\}}$-convolution of $f$ and $g$ exists if $(f(x)\otimes g(y))\varphi(x+y)\in\DD'^{\{M_p\}}_{L^1}(\RR^{2d})$ for all $\varphi\in\DD^{\{M_p\}}(\RR^d)$; in this case $f*g$ is defined by\footnote{The definition is the natural ultradistributional analogue to the definition od $\DD'$-convolution of distributions as introduced by Schwartz.}
$$
\langle f*g,\varphi\rangle:={}_{\DD'^{\{M_p\}}_{L^1}(\RR^{2d})}\langle (f(x)\otimes g(y))\varphi(x+y),1_{x,y}\rangle_{\DD^{\{M_p\}}_{L^{\infty},c}(\RR^{2d})},\quad \varphi\in\DD^{\{M_p\}}(\RR^d),
$$
where $1_{x,y}$ is the function that is identically equal to $1$ on $\RR^{2d}$. If one (or both) of $f$ and $g$ is with compact support, the $\DD'^{\{M_p\}}$-convolution of $f$ and $g$ coincides with the ordinary convolution as defined in \cite{Komatsu1}. We recall the following result which provides equivalent conditions for the existence of $\DD'^{\{M_p\}}$-convolution of two ultradistributions, which, in  practice, are far easier to work with.

\begin{theorem}[{\cite[Theorem 8.2]{D-P-P-V}}]\label{dcon-exi-con}
Let $f,g\in\DD'^{\{M_p\}}(\RR^d)$. The following conditions are equivalent.
\begin{itemize}
\item[$(i)$] The $\DD'^{\{M_p\}}$-convolution of $f$ and $g$ exists.
\item[$(ii)$] For all $\varphi\in\DD^{\{M_p\}}(\RR^d)$, $(\varphi*\check{f})g\in\DD'^{\{M_p\}}_{L^1}(\RR^d)$.
\item[$(iii)$] For all $\varphi\in\DD^{\{M_p\}}(\RR^d)$, $(\varphi*\check{g})f\in\DD'^{\{M_p\}}_{L^1}(\RR^d)$.
\end{itemize}
\end{theorem}

\begin{remark}\label{rem-for-exi-con-alt-def-cor}
If the $\DD'^{\{M_p\}}$-convolution of $f$ and $g$ exists (i.e., if they satisfy the equivalent conditions in Theorem \ref{dcon-exi-con}), then
$$
\langle f*g,\varphi\rangle={}_{\DD'^{\{M_p\}}_{L^1}(\RR^d)}\langle (\varphi*\check{f})g,1\rangle_{\DD^{\{M_p\}}_{L^{\infty},c}(\RR^d)}= {}_{\DD'^{\{M_p\}}_{L^1}(\RR^d)}\langle (\varphi*\check{g})f,1\rangle_{\DD^{\{M_p\}}_{L^{\infty},c}(\RR^d)},\quad \varphi\in\DD^{\{M_p\}}(\RR^d).
$$
We only prove the first identity as the proof of the second is analogous. Pick $\chi\in\DD^{\{M_p\}}(\RR^d)$ such that $0\leq \chi\leq 1$, $\chi(x)=1$ when $|x|\leq 1$ and $\chi(x)=0$ when $|x|\geq 2$. For each $n\in\ZZ_+$, set $\chi_n(x):=\chi(x/n)$. Let $\varphi\in\DD^{\{M_p\}}(\RR^d)$ be arbitrary but fixed. Pick $n_0\in\ZZ_+$ such that $\supp\varphi\subseteq \{x\in\RR^d\,|\, |x|\leq n_0\}$; thus $\chi_n=1$ on $\supp\varphi$ for all $n\geq n_0$. For $n\geq n_0$, we have
\begin{equation}\label{equ-for-con-exi-alt-dfk}
\langle (\varphi*\check{f})g,\chi_n\rangle=\langle (f(x)\otimes g(y))\varphi(x+y),\chi_{n}(x+y)\chi_n(y)\rangle.
\end{equation}
Notice that $\chi_n\rightarrow 1_x$ in $\DD^{\{M_p\}}_{L^{\infty},c}(\RR^d)$ and $\chi_n(x+y)\chi_n(y)\rightarrow 1_{x,y}$ in $\DD^{\{M_p\}}_{L^{\infty},c}(\RR^{2d})$. Hence, the left-hand side of \eqref{equ-for-con-exi-alt-dfk} tends to $\langle (\varphi*\check{f})g,1\rangle$ while the right-hand side tends to $\langle (f(x)\otimes g(y))\varphi(x+y),1_{x,y}\rangle=\langle f*g,\varphi\rangle$ and the proof is complete.
\end{remark}

\subsection{The short-time Fourier transform and the Wigner transform}\label{stft-sub-con-wig}

Let $(M_p)_{p\in\NN}$ be a weight sequence which satisfies $p!^{1/2}\subset M_p$. The short-time Fourier transform (from now on, abbreviated as STFT) of $f\in\SSS'^{\{M_p\}}(\RR^d)$ with window $\varphi\in\SSS^{\{M_p\}}(\RR^d)$ is defined by
\begin{equation}\label{stft1}
V_{\varphi}f(x,\xi):=\langle f,e^{-2\pi i \xi\, \cdot}\overline{\varphi(\cdot-x)}\rangle,\quad x,\xi\in\RR^d;
\end{equation}
$V_{\varphi}f$ is a continuous function on $\RR^{2d}$ with ultrapolynomial growth of class $\{M_p\}$ and the mapping $f\mapsto V_{\varphi}f$, $\SSS'^{\{M_p\}}(\RR^d)\rightarrow \SSS'^{\{M_p\}}(\RR^{2d})$, is continuous and it restricts to a continuous mapping $\psi\mapsto V_{\varphi}\psi$, $\SSS^{\{M_p\}}(\RR^d)\rightarrow \SSS^{\{M_p\}}(\RR^{2d})$ (see \cite[Proposition 2.8 and Proposition 2.9]{and-jasson}). Its adjoint $V^*_{\varphi}:\SSS^{\{M_p\}}(\RR^{2d})\rightarrow \SSS^{\{M_p\}}(\RR^d)$ is given by
$$
V^*_{\varphi}\Psi(x)=\int_{\RR^{2d}}\Psi(y,\eta) \varphi(x-y)e^{2\pi i x\eta}dyd\eta,\quad \Psi\in\SSS^{\{M_p\}}(\RR^{2d}).
$$
It extends to a continuous mapping $V^*_{\varphi}:\SSS'^{\{M_p\}}(\RR^{2d})\rightarrow \SSS'^{\{M_p\}}(\RR^d)$ (see \cite[Proposition 2.8 and Proposition 2.9]{and-jasson}) and $V^*_{\psi}V_{\varphi}=(\psi,\varphi)_{L^2}\operatorname{Id}_{\SSS'^{\{M_p\}}(\RR^d)}$. If both $\varphi$ and $\psi$ are in $L^2(\RR^d)$ then $V_{\varphi}\psi$ can still be defined by \eqref{stft1} and in this case $\psi\mapsto V_{\varphi}\psi$, $L^2(\RR^d)\rightarrow L^2(\RR^{2d})$, is well-defined and continuous.\\
\indent Recall that the Wigner transform $W(f,g)$ of $f,g\in L^2(\RR^d)$ is given by
$$
W(f,g)(x,\xi)=\int_{\RR^d}e^{-2\pi i y\xi}
f(x+y/2)\overline{g(x-y/2)}dy,\quad x,\xi\in\RR^d;
$$
$W(f,g)\in L^2(\RR^{2d})$ and, in fact, $W(f,g)(x,\xi)=2^d e^{4\pi i x\xi}V_{\check{g}}f(2x,2\xi)$, $x,\xi\in\RR^d$ (see \cite[Lemma 4.3.1, p. 64]{G}). Consequently, $W(\varphi,\psi)\in \SSS^{\{M_p\}}(\RR^{2d})$ when $\varphi,\psi\in\SSS^{\{M_p\}}(\RR^d)$.\\
\indent Given $a\in\SSS'^{\{M_p\}}(\RR^{2d})$, we denote by $a^w$ the Weyl quantisation of $a$, i.e. the pseudo-differential operator $a^w:\SSS^{\{M_p\}}(\RR^d)\rightarrow \SSS'^{\{M_p\}}(\RR^d)$ defined by
$$
\langle a^w\psi,\theta\rangle=\langle \mathcal{F}^{-1}_{\xi\rightarrow x-y} a((x+y)/2,\xi),\theta(x)\otimes \psi(y)\rangle,\quad \psi,\theta\in\SSS^{\{M_p\}}(\RR^d).
$$
We recall the following important identity
$$
\langle a^w\psi,\theta\rangle=\langle a, W(\psi,\overline{\theta})\rangle,\quad \psi,\theta\in\SSS^{\{M_p\}}(\RR^d),\, a\in\SSS'^{\{M_p\}}(\RR^{2d}).
$$

\section{Bounds on the Wigner transform of ultradifferentiable functions with super-exponential decay}\label{wigner-tra-ult-dif-rdc}

Throughout the section, we will need the following inequalities:
\begin{gather}
|x|^{k_1} |y|^{k_2}\leq 2^{k_1+k_2}(|x-y/2|^{k_1+k_2}+|x+y/2|^{k_1+k_2}),\quad x,y\in\RR^d,\, k_1,k_2\in\NN;\label{ine-for-prod-abs} \\
e^{-r\langle x-y/2\rangle^q-r\langle x+y/2\rangle^q}\leq e^{-2r\langle x\rangle^q},\quad  x,y\in\mathbb R^d,\, q\geq 1,\, r\geq 0.\label{x-t2}
\end{gather}
We only prove the second as the proof of the first inequality is straightforward. Since the function $\rho\mapsto(1+\rho^2)^{q/2}$, $[0,\infty)\rightarrow [1,\infty)$, is convex and monotonically increasing when $q\geq 1$, the validity of \eqref{x-t2} follows from
$$
\frac{\langle x-y/2\rangle^q+\langle x+y/2\rangle^q}{2}\geq \left\langle \frac{|x-y/2|+|x+y/2|}{2}\right\rangle^q\geq \langle |x|\rangle^q=\langle x\rangle^q.
$$

\begin{proposition}\label{ex1}
Let $(M_p)_{p\in\NN}$ be a weight sequence that satisfies $p!\subset M_p$ and let $q\geq 1$ and $r>0$. Let $\varphi,\psi\in\mathcal{C}^{\infty}(\RR^d)$ satisfy the following: for every $r'\in(0,r)$ there is $h'>0$ such that
\begin{equation}\label{equ-for-est-for-res}
\sup_{\alpha\in\NN^d}\frac{h'^{|\alpha|}\|e^{r'\langle \cdot\rangle^q}\partial^{\alpha}\varphi\|_{L^{\infty}(\RR^d)}}{M_{\alpha}}<\infty\quad \mbox{and}\quad \sup_{\alpha\in\NN^d}\frac{h'^{|\alpha|}\|e^{r'\langle \cdot\rangle^q}\partial^{\alpha}\psi\|_{L^{\infty}(\RR^d)}}{M_{\alpha}}<\infty.
\end{equation}
Then, for every $r'\in(0,r)$ there is $h'>0$ such that
\begin{equation}\label{equ-claim-for-est-wigner-doub}
\sup_{\alpha,\beta,\gamma,\delta\in\NN^d} \sup_{x,\xi\in\RR^d}\frac{h'^{|\alpha+\beta+\gamma+\delta|}|x^{\gamma}\xi^{\delta} \partial^{\alpha}_x\partial^{\beta}_{\xi} W(\varphi,\psi)(x,\xi)|e^{2r'\langle x\rangle^q}}{M_{\alpha}\beta!^{1/q}\gamma!^{1/q}M_{\delta}}<\infty.
\end{equation}
Furthermore, for every $r'\in(0,r)$ there are $c',h'>0$ such that
\begin{equation}\label{inequal-sec-est0-wigner}
\sup_{\alpha,\beta\in\NN^d} \sup_{x,\xi\in\RR^d}\frac{h'^{|\alpha+\beta|} |\partial^{\alpha}_x\partial^{\beta}_{\xi} W(\varphi,\psi)(x,\xi)|e^{2r'\langle x\rangle^q}e^{M(c'|\xi|)}}{M_{\alpha}\beta!^{1/q}}<\infty.
\end{equation}
\end{proposition}

\begin{proof} The bound \eqref{inequal-sec-est0-wigner} is equivalent to \eqref{equ-claim-for-est-wigner-doub}. Indeed, \eqref{equ-claim-for-est-wigner-doub} gives
$$
\frac{(h'/2)^{|\alpha+\beta+\gamma+\delta|}|x^{\gamma}\xi^{\delta} \partial^{\alpha}_x\partial^{\beta}_{\xi} W(\varphi,\psi)(x,\xi)|e^{2r'\langle x\rangle^q}}{M_{\alpha}!\beta!^{1/q}\gamma!^{1/q}M_{\delta}}\leq \frac{C}{2^{|\gamma|+|\delta|}}.
$$
Summing these inequalities over $\gamma,\delta\in\NN^d$ gives \eqref{inequal-sec-est0-wigner} since (recall $|\kappa|!\leq d^{|\kappa|}\kappa!$, $\forall \kappa\in\NN^d$)
$$
\sum_{\delta\in\NN^d}\frac{h'^{|\delta|}|\xi^{\delta}|}{2^{|\delta|}M_{\delta}}\geq \sum_{k=0}^{\infty}\frac{h'^k}{2^kd^kM_k}\sum_{|\delta|=k}\frac{k!}{\delta!}|\xi^{\delta}| \geq \sum_{k=0}^{\infty}\frac{h'^k|\xi|^k}{2^kd^kM_k}\geq e^{M(h'|\xi|/(2d))}.
$$
Clearly, \eqref{inequal-sec-est0-wigner} implies \eqref{equ-claim-for-est-wigner-doub} with $r'-\varepsilon$ in place of $r'$ for arbitrary small $\varepsilon>0$.\\
\indent To prove \eqref{equ-claim-for-est-wigner-doub}, let $\tilde{r}\in(0,r)$ be arbitrary but fixed and let $h'>0$ be such that \eqref{equ-for-est-for-res} holds true with $r':=(r+\tilde{r})/2\in(0,r)$. We estimate as follows
\begin{align}
|x^{\gamma}\xi^{\delta} &\partial^{\alpha}_x\partial^{\beta}_{\xi} W(\varphi,\psi)(x,\xi)|\nonumber\\
&= (2\pi)^{|\beta|-|\delta|} \sum_{\alpha'+\alpha''=\alpha}{\alpha\choose\alpha'} \sum_{\substack{\delta'+\delta''+\delta'''=\delta\\ \delta'''\leq \beta}} {\delta\choose {\delta',\delta'',\delta'''}}\frac{\beta!}{(\beta-\delta''')!}\nonumber\\
&\quad\cdot\left|\int_{\RR^d}x^{\gamma}e^{-2\pi i y\xi} y^{\beta-\delta'''}\partial^{\delta'}_y\partial^{\alpha'}_x(\varphi(x+y/2)) \partial^{\delta''}_y\partial^{\alpha''}_x(\overline{\psi(x-y/2)})dy\right|\nonumber\\
&\leq C_1 2^{|\alpha|}(4\pi)^{|\beta|} \sum_{\substack{\delta'+\delta''+\delta'''=\delta\\ \delta'''\leq \beta}} {\delta\choose {\delta',\delta'',\delta'''}}\delta'''! h'^{-|\alpha|-|\delta|+|\delta'''|} M_{\alpha+\delta-\delta'''}\nonumber\\
&\quad\cdot\int_{\RR^d}|x|^{|\gamma|} |y|^{|\beta-\delta'''|}e^{-r'\langle x+y/2\rangle^q}e^{-r'\langle x-y/2\rangle^q}dy.\label{est-der-wig-for-first-psk}
\end{align}
Set $r_1:=\min\{1,(r-\tilde{r})/2\}>0$. In view of \eqref{ine-for-prod-abs}, \eqref{x-t2} and the inequality $\lambda^k\leq k!^{1/q}e^{\lambda^q/q}$, $\lambda\geq 0$, $k\in\NN$, we infer
\begin{align*}
|x|^{|\gamma|} &|y|^{|\beta-\delta'''|}e^{-r'\langle x+y/2\rangle^q}e^{-r'\langle x-y/2\rangle^q}\\
&\leq 2^{|\gamma+\beta-\delta'''|} e^{-2\tilde{r}\langle x\rangle^q} (|x+y/2|^{|\gamma+\beta-\delta'''|}e^{-r_1\langle x+y/2\rangle^q} +|x-y/2|^{|\gamma+\beta-\delta'''|}e^{-r_1\langle x-y/2\rangle^q})\\
&\leq 2^{|\gamma+\beta-\delta'''|} e^{-2\tilde{r}\langle x\rangle^q} (2/r_1)^{|\gamma+\beta-\delta'''|}|\gamma+\beta-\delta'''|!^{1/q} (e^{-\frac{r_1}{2}\langle x+y/2\rangle^q} +e^{-\frac{r_1}{2}\langle x-y/2\rangle^q}).
\end{align*}
As $|\kappa|!\leq d^{|\kappa|}\kappa!$, $\forall \kappa\in\NN^d$, and $p!\leq C'' L^pM_p$ for some $C'',L\geq 1$ (since $p!\subset M_p$), we deduce
\begin{align}
|x^{\gamma}&\xi^{\delta} \partial^{\alpha}_x\partial^{\beta}_{\xi} W(\varphi,\psi)(x,\xi)|\nonumber\\
&\leq C_2e^{-2\tilde{r}\langle x\rangle^q} (4\pi)^{|\beta|}2^{|\alpha|+|\beta|+|\gamma|}d^{|\beta|+|\gamma|}  h'^{-|\alpha|-|\delta|}(2/r_1)^{|\gamma|+|\beta|}M_{\alpha+\delta}\nonumber\\
&\quad \cdot\sum_{\substack{\delta'+\delta''+\delta'''=\delta\\ \delta'''\leq \beta}} {\delta\choose {\delta',\delta'',\delta'''}} 4^{-|\delta'''|}(r_1h'L)^{|\delta'''|} (\gamma+\beta-\delta''')!^{1/q}\label{est-for-c-add-ass}\\
&\leq C_3 e^{-2\tilde{r}\langle x\rangle^q} (4\pi)^{|\beta|}2^{|\alpha|+2|\beta|+2|\gamma|}(H/h')^{|\alpha|+|\delta|} (2d/r_1)^{|\gamma|+|\beta|}(2+r_1h'L/4)^{|\delta|} M_{\alpha}\beta!^{1/q}\gamma!^{1/q}M_{\delta},\label{est-for-con-add-assump-on-seq}
\end{align}
which implies the validity of \eqref{equ-claim-for-est-wigner-doub}.
\end{proof}

\begin{remark}
The above proposition is applicable with $\varphi=\psi=e^{-r\langle\cdot\rangle^q}\in \SSS^{\{p!\}}_{\{p!^{1/q}\}}(\RR^d)$, $r>0$, $q\geq 1$, in view of the following bounds \cite[Corollary 3.4]{PPV}: for every $q>0$ and $r\in\RR\backslash\{0\}$ there is $C\geq 1$ such that
\begin{equation}\label{bound-for-exp-der}
|\partial^{\alpha}(e^{r\langle x\rangle^q})|\leq C^{|\alpha|}\alpha!e^{r\langle x\rangle^q +\langle x\rangle^{q-1}},\quad \forall x\in\RR^d,\, \forall \alpha\in\NN^d.
\end{equation}
In this case, the best possible constant $r>0$ for which the assumption in Proposition \ref{ex1} holds is the constant $r$ that appears in $e^{-r\langle \cdot\rangle^q}$ since $|\partial_{x_1}(e^{-r\langle x\rangle^q})|=rq|x_1|\langle x\rangle^{q-2}e^{-r\langle x\rangle^q}$ and the latter grows faster then $e^{-r\langle x\rangle^q}$ along the $x_1$-axis.\\
\indent Notice that the assumption in the proposition is satisfied for any $\varphi,\psi\in\SSS^{\{M_p\}}_{\{p!^{1/q}\}}(\RR^d)$ for some $r>0$.
\end{remark}

\begin{remark}\label{rem-for-ext-val-pro-est}
By carefully inspecting the proof of Proposition \ref{ex1}, one easily sees that the proposition is valid if one only assumes that $(M_p)_{p\in\NN}$ is a weight sequence which satisfies $p!^{1-1/q}\subset M_p$ when $q>1$. In fact, the proof is the same the only difference being that in the estimates \eqref{est-for-c-add-ass} and \eqref{est-for-con-add-assump-on-seq} instead of using
$$
\delta'''!\leq C''L^{|\delta'''|}M_{\delta'''}\quad \mbox{and}\quad (\gamma+\beta-\delta''')!^{1/q}\leq (\gamma+\beta)!^{1/q}
$$
one employs the bounds
$$
\delta'''!\leq C''L^{|\delta'''|}M_{\delta'''}\delta'''!^{1/q}\quad \mbox{and}\quad (\gamma+\beta-\delta''')!^{1/q}\leq (\gamma+\beta)!^{1/q}\delta'''!^{-1/q}.
$$
\end{remark}

As an immediate consequence of Proposition \ref{ex1}, we have the following result.

\begin{corollary}\label{posebno}
Let $(M_p)_{p\in\NN}$, $q\geq 1$, $r>0$ and $\varphi,\psi\in\mathcal{C}^{\infty}(\RR^d)$ be as in Proposition \ref{ex1}. Let $P_1$ and $P_2$ be ultradifferential operators of class $\{M_p\}$ and $\{p!^{1/q}\}$ respectively. Then for every $r'\in(0,r)$ there are $h',c'>0$ such that
\begin{equation}
\sup_{\alpha,\beta\in\NN^d}\sup_{x,\xi\in\RR^d}\frac{h'^{|\alpha+\beta|}} {M_{\alpha}\beta!^{1/q}} |D^{\alpha}_xD^{\beta}_{\xi}P_1(D_x)P_2(D_{\xi}) W(\varphi,\psi)(x,\xi)|e^{2r'\langle x\rangle^q}e^{M(c'|\xi|)}<\infty.
\end{equation}
\end{corollary}

In order to consider test functions which decrease more rapid than $e^{-l\langle x\rangle^q}$ as $|x|\to\infty$, we proceed as follows.

\begin{proposition}\label{ex111}
Let $(M_p)_{p\in\NN}$ be a weight sequence that satisfies $p!\subset M_p$ and let $q\geq 1$. Let $f\in\mathcal{C}(\RR^d)\cap L^{\infty}(\RR^d)$ be non-negative and monotonically decreasing with respect to $|x|$, i.e. $|x|\leq |y|$ implies $f(x)\geq f(y)$. Let $\varphi,\psi\in\mathcal{C}^{\infty}(\RR^d)$ satisfy the following: there are $h',C'>0$ such that
\begin{equation}
\frac{h'^{|\alpha|}e^{h'\langle x\rangle^q} |\partial^{\alpha}\varphi(x)|}{M_{\alpha}}\leq C'f(x)\quad \mbox{and}\quad \frac{h'^{|\alpha|}e^{h'\langle x\rangle^q} |\partial^{\alpha}\psi(x)|}{M_{\alpha}}\leq C'f(x),
\end{equation}
for all $x\in\RR^d$, $\alpha\in\NN^d$. Then there are $h,C>0$ such that
\begin{equation}\label{est-wigner-for-mon-decf}
\frac{h^{|\alpha+\beta+\gamma+\delta|}|x^{\gamma}\xi^{\delta} \partial^{\alpha}_x\partial^{\beta}_{\xi}W(\varphi,\psi)(x,\xi)|} {M_{\alpha}\beta!^{1/q}\gamma!^{1/q}M_{\delta}}\leq Cf(x),\quad \mbox{for all}\,\, x,\xi\in\RR^d,\, \alpha,\beta,\gamma,\delta\in\NN^d.
\end{equation}
Furthermore, there are $h,c,C>0$ such that
\begin{equation}\label{est-wigne-sec-for-decf}
\frac{h^{|\alpha+\beta|} |\partial^{\alpha}_x\partial^{\beta}_{\xi}W(\varphi,\psi)(x,\xi)|} {M_{\alpha}\beta!^{1/q}} \leq Ce^{-c\langle x\rangle^{q}}e^{-M(c|\xi|)} f(x),\quad \mbox{for all}\,\, x,\xi\in\RR^d,\, \alpha,\beta\in\NN^d.
\end{equation}
\end{proposition}

\begin{proof} Similarly as in the first part of the proof of Proposition \ref{ex1}, one shows that \eqref{est-wigner-for-mon-decf} is equivalent to \eqref{est-wigne-sec-for-decf}. The proof of \eqref{est-wigner-for-mon-decf} is analogous to the proof of Proposition \ref{ex1} and we omit it. The only difference is that in \eqref{est-der-wig-for-first-psk} now there is the extra term $f(x+y/2)f(x-y/2)$ which can be bounded as follows. At least one of $|x+y/2|$ and $|x-y/2|$ must be greater or equal to $|x|$ since $\max\{|x+y/2|,|x-y/2|\}<|x|$ leads to a contradiction and, consequently, the properties of $f$ imply $f(x+y/2)f(x-y/2)\leq C''f(x)$, for all $x,y\in\RR^d$.
\end{proof}

\begin{remark}
Similarly as in Remark \ref{rem-for-ext-val-pro-est}, Proposition \ref{ex111} is valid if one only assumes that $(M_p)_{p\in\NN}$ is a weight sequence that satisfies $p!^{1-1/q}\subset M_p$ when $q>1$.
\end{remark}

\begin{corollary}\label{iter2}
Let $(M_p)_{p\in\NN}$, $q\geq 1$, $f$ and $\varphi,\psi\in\mathcal{C}^{\infty}(\RR^d)$ be as in Proposition \ref{ex111}. Let $P_1$ and $P_2$ be ultradifferential operators of class $\{M_p\}$ and $\{p!^{1/q}\}$ respectively. Then there are $h,c,C>0$ such that
\begin{equation}
\frac{h^{|\alpha+\beta|}} {M_{\alpha}\beta!^{1/q}} |D^{\alpha}_xD^{\beta}_{\xi}P_1(D_x)P_2(D_{\xi}) W(\varphi,\psi)(x,\xi)|\leq Ce^{-c\langle x\rangle^q}e^{-M(c|\xi|)}f(x),
\end{equation}
for all $x,\xi\in\RR^d$, $\alpha,\beta\in\NN^d$.
\end{corollary}

\subsection{Example}\label{extra}
Consider $\varphi(x)=\exp(-te^{\langle x\rangle^q})$, $x\in\RR^d$, with $t>0$ and $q>0$. Our goal is to estimate the derivatives of $\varphi$. We apply the Fa\'a di Bruno formula \cite[Corollary 2.10]{Faa} to the composition of $e^{-t\lambda}$, $\lambda>0$, with $e^{\langle\cdot\rangle^q}$. Using the formulation as in \cite[Corollary 2.10]{Faa} and employing \eqref{bound-for-exp-der} (with $r=1$), we infer the following bounds for any $\rho>0$ and $\alpha\in\NN^d\backslash\{0\}$
\begin{align*}
\frac{|\partial^{\alpha}\varphi(x)|}{|\alpha|!}&\leq \sum_{m=1}^{|\alpha|}t^m \exp(-te^{\langle x\rangle^q})\sum_{p(\alpha,m)}\prod_{j=1}^{|\alpha|}\frac{C^{k_j|\alpha^{(j)}|}e^{k_j\langle x\rangle^q+k_j\langle x\rangle^{q-1}}}{k_j!}\\
&\leq C^{|\alpha|}|\alpha|!^{\rho}\exp(-te^{\langle x\rangle^q}) \sum_{m=1}^{|\alpha|}\frac{t^m e^{m\langle x\rangle^q+m\langle x\rangle^{q-1}}}{m!^{1+\rho}}\cdot m!\sum_{p(\alpha,m)}\prod_{j=1}^{|\alpha|}\frac{1}{k_j!}\\
&\leq C^{|\alpha|}|\alpha|!^{\rho}\exp\left(-te^{\langle x\rangle^q}+(1+\rho)t^{1/(1+\rho)}e^{(1+\rho)^{-1}(\langle x\rangle^q+\langle x\rangle^{q-1})}\right) \sum_{m=1}^{|\alpha|} m!\sum_{p(\alpha,m)}\prod_{j=1}^{|\alpha|}\frac{1}{k_j!}.
\end{align*}
Employing \cite[Lemma 7.4]{PP3}, we infer
$$
\sum_{m=1}^{|\alpha|}m!\sum_{p(\alpha,m)}\prod_{j=1}^{|\alpha|}\frac{1}{k_j!} \leq 2^{|\alpha|(d+1)},
$$
hence
\begin{equation}
|\partial^{\alpha}\varphi(x)|\leq (2^{d+1}C)^{|\alpha|}|\alpha|!^{1+\rho}\exp\left(-te^{\langle x\rangle^q}+(1+\rho)t^{1/(1+\rho)}e^{(1+\rho)^{-1}(\langle x\rangle^q+\langle x\rangle^{q-1})}\right).
\end{equation}
Notice that for every $\varepsilon>0$ there exists $C'=C'(\varepsilon,t,\rho)\geq 1$ such that
$$
(1+\rho)t^{1/(1+\rho)}e^{(1+\rho)^{-1}(\langle x\rangle^q+\langle x\rangle^{q-1})}\leq \varepsilon e^{\langle x\rangle^q},\quad \mbox{for all}\,\, |x|\geq C'.
$$
Consequently, we obtained the following result.

\begin{proposition}\label{121}
Let $\varphi(x)=\exp(-te^{\langle x\rangle^q})$, with $t>0$ and $q>0$. For every $\rho>0$ and $\tau\in(0,t)$ there exists $C\geq 1$ such that
$$
|\partial^{\alpha}\varphi(x)|\leq C^{|\alpha|}|\alpha|!^{1+\rho} \exp(-\tau e^{\langle x\rangle^q}),\quad \forall x\in\RR^d,\, \forall \alpha\in\NN^d.
$$
\end{proposition}

In view of Proposition \ref{121}, Proposition \ref{ex111} is applicable for $\varphi(x)=\psi(x)=\exp(-te^{\langle x\rangle^s})$, $x\in\RR^d$, with $t>0$ and $s>0$, $f(x)=\exp(-\tau e^{\langle x\rangle^s})$, $\tau\in(0,t)$, and $M_p=p!^{1+\rho}$, $\rho>0$.

\section{Localisation operators with super-exponentially bounded symbols}

\subsection{Localisation operators}

Let $(M_p)_{p\in\NN}$ be a weight sequences which satisfies $p!^{1/2}\subset M_p$. The localisation operator $A_a^{\varphi_1,\varphi_2}$ with symbol $a\in\SSS'^{\{M_p\}}(\RR^{2d})$ and windows $\varphi_1,\varphi_2\in\SSS^{\{M_p\}}(\RR^d)$ is defined by
\begin{equation}\label{av}
A^{\varphi_1,\varphi_2}_a:\SSS^{\{M_p\}}(\RR^d)\rightarrow \SSS'^{\{M_p\}}(\RR^d),\, A^{\varphi_1,\varphi_2}_a \psi= V^*_{\varphi_2}(aV_{\varphi_1}\psi);
\end{equation}
when $\varphi_1=\varphi_2=:\varphi$, we employ the shorthand $A^{\varphi}_a$. The continuity properties of the STFT and its adjoint we recalled in Subsection \ref{stft-sub-con-wig} imply that $A^{\varphi_1,\varphi_2}_a$ is well-defined and continuous. Notice that $\langle A^{\varphi_1,\varphi_2}_a\psi,\theta\rangle=\langle a, V_{\varphi_1}\psi \overline{V_{\varphi_2}\overline{\theta}}\rangle$, $\psi,\theta\in\SSS^{\{M_p\}}(\RR^d)$. In this case (see \cite[Theorem 5.2]{teofanov})
\begin{equation}\label{equ-rel-loc-weyl-quat-eqt}
A^{\varphi_1,\varphi_2}_a=b^w,\quad \mbox{with}\quad b:=a*W(\varphi_2,\varphi_1).
\end{equation}
Assume now that $(M_p)_{p\in\NN}$ additionally satisfies $(M.3)$ and $\varphi_1,\varphi_2\in\SSS^{\{M_p\}}_{\{p!^{1/q}\}}(\RR^d)$ with $q\geq 1$. We can immediately conclude that $A^{\varphi_1,\varphi_2}_a:\SSS^{\{M_p\}}(\RR^d)\rightarrow \SSS'^{\{M_p\}}(\RR^d)$ is well-defined and continuous when $a\in\SSS'^{\{M_p\}}(\RR^{2d})$ (as $p!^{1/q}\subset M_p$). Our goal is to extend the definition of $A^{\varphi_1,\varphi_2}_a$ for a larger class of symbols $a$ by allowing it to be a mapping from $\DD^{\{M_p\}}(\RR^d)$ into $\DD'^{\{M_p\}}(\RR^d)$ but keeping the relation \eqref{equ-rel-loc-weyl-quat-eqt} where the convolution in \eqref{equ-rel-loc-weyl-quat-eqt} should be understood in $\DD'^{\{M_p\}}(\RR^{2d})$-sense.

\begin{remark}\label{rem-for-res-ajr}
If $(M_p)_{p\in\NN}$, $\varphi_1$ and $\varphi_2$ are as above, by employing \cite[Proposition 2.8, Proposition 2.9]{and-jasson}, one can extend $A^{\varphi_1,\varphi_2}_a$ to a continuous operator from $\DD^{\{M_p\}}(\RR^d)$ into $\DD'^{\{M_p\}}(\RR^d)$ when $a$ belongs to
\begin{equation}\label{set-from-ord-the-a}
\left(\SSS'^{\{M_p\}}_{\{p!^{1/q}\}}(\RR^d)\hat{\otimes} \SSS'^{\{p!^{1/q}\}}_{\{M_p\}}(\RR^d)\right)\cap \DD'^{\{M_p\}}(\RR^{2d})
\end{equation}
provided that the $\DD'^{\{M_p\}}$-convolution of $a$ and $W(\varphi_2,\varphi_1)$ exists and the Weyl quantisation of the resulting symbol $b$ (given by \eqref{equ-rel-loc-weyl-quat-eqt}) is a continuous operator from $\DD^{\{M_p\}}(\RR^d)$ into $\DD'^{\{M_p\}}(\RR^d)$. The estimates we proved in the previous section will allows us to define localisation operators with symbols that go beyond \eqref{set-from-ord-the-a}.
\end{remark}

\subsection{Extensions}

Throughout the rest of the article, $(M_p)_{p\in\NN}$ will be a weight sequence that additionally satisfies $(M.3)$.\\
\indent We start by defining the following class of subspaces of $\DD'^{\{M_p\}}(\RR^{2d})$. For $q\geq 1$ and $\varphi_1,\varphi_2\in\SSS^{\{M_p\}}_{\{p!^{1/q}\}}(\RR^d)$, we denote by $\mathcal{A}_{\varphi_1,\varphi_2}$ the space of all $a\in\DD'^{\{M_p\}}(\RR^{2d})$ such that the $\DD'^{\{M_p\}}$-convolution of $a$ and $W(\varphi_2,\varphi_1)$ exists and
\begin{equation}\label{equ-for-wey-from-loc}
b:=a*W(\varphi_2,\varphi_1)\in\DD'^{\{M_p\}}(\RR^{2d})
\end{equation}
satisfies the following:
\begin{itemize}
\item[$(i)$] $b\in \mathcal{C}(\RR^{2d})$,
\item[$(ii)$] for each $\xi\in\RR^d$, $b(\cdot,\xi)\in\mathcal{C}^{\infty}(\RR^d)$ and for every $\alpha\in\NN^d$, the function $(x,\xi)\mapsto D^{\alpha}_x b(x,\xi)$ is continuous on $\RR^{2d}$,
\item[$(iii)$] for every compact subset $K$ of $\RR^d$ there exist $(k_p)\in\mathfrak{R}$, $h>0$ and $C>0$  such that
\begin{equation}\label{gelfand1}
\frac{h^\alpha}{M_{\alpha}}|D^{\alpha}_x b(x,\xi)|\leq C  e^{N_{k_p}(|\xi|)},\quad \forall x\in K,\, \forall\xi\in\RR^d,\, \forall\alpha\in\NN^d.
\end{equation}
\end{itemize}

Employing the same technique as in the proof of \cite[Theorem 5.1]{PP}, one can show the following result.

\begin{proposition}\label{pre}
Let $q\geq 1$ and $\varphi_1,\varphi_2\in\SSS^{\{M_p\}}_{\{p!^{1/q}\}}(\RR^d)$. Let $a\in\mathcal{A}_{\varphi_1,\varphi_2}$ and let $b$ be defined by \eqref{equ-for-wey-from-loc}. Then, for each $\chi\in\mathcal D^{\{M_p\}}(\RR^{2d})$, the integral
\begin{equation}\label{osci}
\int_{\mathbb R^{d}}\int_{\RR^d}\int_{\RR^d}e^{2\pi i(x-y)\xi} b\left(\frac{x+y}{2},\xi\right)\chi(x,y)dxdyd\xi
\end{equation}
is well-defined as oscillatory integral and
\begin{equation}\label{ker-ult-giv}
\chi\mapsto \int_{\mathbb R^{d}}\int_{\RR^d}\int_{\RR^d}e^{2\pi i(x-y)\xi} b\left(\frac{x+y}{2},\xi\right)\chi(x,y)dxdyd\xi,\quad \DD^{\{M_p\}}(\RR^{2d})\rightarrow \CC,
\end{equation}
is a well-defined element of $\DD'^{\{M_p\}}(\RR^{2d})$.
\end{proposition}

Denoting by $K\in\DD'^{\{M_p\}}(\RR^{2d})$ the ultradistribution \eqref{ker-ult-giv}, Proposition \ref{pre} verifies that
$$
b^w:\DD^{\{M_p\}}(\RR^d)\rightarrow \DD'^{\{M_p\}}(\RR^d),\quad \langle b^w \psi,\theta\rangle=\langle K,\theta\otimes \psi\rangle,\,\,\psi,\theta\in\DD^{\{M_p\}}(\RR^d),
$$
is well-defined and continuous (i.e. the Weyl quantisation of $b$ with kernel $K$). Given $q\geq1$, $\varphi_1,\varphi_2\in\SSS^{\{M_p\}}_{\{p!^{1/q}\}}(\RR^d)$ and $a\in\mathcal{A}_{\varphi_1,\varphi_2}$, we define the localisation operator $A^{\varphi_1,\varphi_2}_a:\DD^{\{M_p\}}(\RR^d)\rightarrow \DD'^{\{M_p\}}(\RR^d)$ by $A^{\varphi_1,\varphi_2}_a:=b^w$ with $b$ given by \eqref{equ-for-wey-from-loc}. It is straightforward to verify that $\SSS'^{\{M_p\}}(\RR^{2d})\subseteq \mathcal{A}_{\varphi_1,\varphi_2}$ (since $W(\varphi_2,\varphi_1)\in\SSS^{\{M_p\}}(\RR^{2d})$) and that the above definition of $A^{\varphi_1,\varphi_2}_a$, $a\in\SSS'^{\{M_p\}}(\RR^{2d})$, coincides with the standard one \eqref{av} (cf. \eqref{equ-rel-loc-weyl-quat-eqt}). Our goal is to find ultradistributions $a\in\mathcal{A}_{\varphi_1,\varphi_2}$ which do not belong to $\SSS'^{\{M_p\}}(\RR^{2d})$. For such symbols $a$, Proposition \ref{pre} states that the operators $A^{\varphi_1,\varphi_2}_a:\DD^{\{M_p\}}(\RR^d)\rightarrow \DD'^{\{M_p\}}(\RR^d)$ are well-defined and continuous and they are non-trivial extensions of the classical localisation operators.

\begin{theorem}\label{ext22}
Let $\varphi_1,\varphi_2\in\SSS^{\{M_p\}}_{\{p!^{1/q}\}}(\RR^d)$, $q\geq 1$, satisfy the assumptions in Proposition \ref{ex1} for some $r>0$. Let $f_j\in L^1_{\operatorname{loc}}(\RR^d)$, $j=1,\ldots,k$, be such that $e^{-l\langle\cdot\rangle^q}f_j\in L^{\infty}(\RR^d)$, $j=1,\ldots,k$, for some $l\in(0,2r)$. Then for any $g_j\in\SSS'^{\{M_p\}}(\RR^d)$, $j=1,\ldots,k$, and ultradifferential operators $P_j$, $j=1,\ldots,k$, of class $\{M_p\}$, the ultradistribution
$$
a(x,\xi)= \sum_{j=1}^k(P_j(D_x)f_j(x))\otimes g_j(\xi)
$$
belongs to $\mathcal{A}_{\varphi_1,\varphi_2}$.
\end{theorem}

\begin{proof} Clearly, it suffices to show the claim when $k=1$, i.e. when $a(x,\xi)=(P(D_x)f(x))\otimes g(\xi)$. We claim that for any $\chi\in\DD^{\{M_p\}}(\RR^{2d})$,
\begin{equation}\label{tos-for-exi-con-dd}
(\check{a}*\chi)W(\varphi_2,\varphi_1)\in L^1(\RR^{2d}).
\end{equation}
In view of \cite[Proposition 2.5]{PPV-JMPA}, there is an ultradifferential operator $\tilde{P}$ of class $\{M_p\}$ and a continuous function $\tilde{g}$ on $\RR^d$ with ultrapolynomial growth of class $\{M_p\}$ such that $g=\tilde{P}(D)\tilde{g}$. Set $\chi_1(x,\xi):=P(-D_x)\tilde{P}(-D_{\xi})\chi(x,\xi)$; clearly $\chi_1\in\DD^{\{M_p\}}(\RR^{2d})$. Then, by employing \eqref{ine-for-seq-assoc} and the inequality \cite[Lemma 3.5]{PPV}
\begin{equation}\label{est-otr-papdjst}
s\langle x-y\rangle^q\leq s\langle x\rangle^q+ q2^{q-1}|s||y|\langle x\rangle^{q-1}+q2^{q-1}|s||y|\langle y\rangle^{q-1},\quad \mbox{for all}\,\, x,y\in\RR^d,\, s\in\RR\backslash\{0\},
\end{equation}
for any $h>0$ and $\varepsilon>0$ we infer
\begin{align*}
|\check{a}*\chi(x,\xi)|&=\left|\int_{\RR^{2d}}f(y) \tilde{g}(\eta)\chi_1(x+y,\xi+\eta)dyd\eta\right|\\
&\leq C_1\int_{\RR^{2d}}e^{l\langle y-x\rangle^q} e^{M(h|\eta-\xi|)}|\chi_1(y,\eta)|dyd\eta\\
&\leq C_2e^{l\langle x\rangle^q}e^{M(2h|\xi|)}\int_{\supp\chi_1}e^{q2^{q-1}l|y|\langle x\rangle^{q-1}+q2^{q-1}l\langle y\rangle^q} e^{M(2h|\eta|)}dyd\eta\\
&\leq C_3e^{(l+\varepsilon)\langle x\rangle^q}e^{M(2h|\xi|)}.
\end{align*}
Now the validity of \eqref{tos-for-exi-con-dd} follows from Proposition \ref{ex1}. Hence, Theorem \ref{dcon-exi-con} implies that the $\DD'^{\{M_p\}}$-convolution of $a$ and $W(\varphi_2,\varphi_1)$ exists. By direct inspection, one verifies that (cf. Remark \ref{rem-for-exi-con-alt-def-cor})
$$
a*W(\varphi_2,\varphi_1)(x,\xi)=\int_{\RR^{2d}}f(y)\tilde{g}(\eta) P(D_x)\tilde{P}(D_{\xi})W(\varphi_2,\varphi_1)(x-y,\xi-\eta)dyd\eta;
$$
the latter is a well-defined $\mathcal{C}^{\infty}$ function in view of Corollary \ref{posebno} and \eqref{est-otr-papdjst}. Corollary \ref{posebno} together with the fact $e^{-N_{r_p}(|\cdot|)}\tilde{g}\in L^{\infty}(\RR^d)$, for some $(r_p)\in\mathfrak{R}$, (cf. \eqref{equ-for-con-equ-for-gro}) also proves that $a*W(\varphi_2,\varphi_1)$ satisfies the bounds \eqref{gelfand1} and the proof of the theorem is complete.
\end{proof}

\begin{remark}
The ultradistribution $a$, in general, may not belong to the space \eqref{set-from-ord-the-a}. For example $a(x,\xi):=e^{l\langle x\rangle^q}\otimes g(\xi)$, with $l\in(0,2r)$ and $g\in\SSS'^{\{M_p\}}(\RR^d)\backslash\{0\}\subseteq \SSS'^{\{p!^{1/q}\}}_{\{M_p\}}(\RR^d)\backslash\{0\}$, does not belong to \eqref{set-from-ord-the-a}. To see this, assume
$$
a\in \SSS'^{\{M_p\}}_{\{p!^{1/q}\}}(\RR^d)\hat{\otimes} \SSS'^{\{p!^{1/q}\}}_{\{M_p\}}(\RR^d)=(\SSS^{\{M_p\}}_{\{p!^{1/q}\}}(\RR^d)\hat{\otimes} \SSS^{\{p!^{1/q}\}}_{\{M_p\}}(\RR^d))'.
$$
Set $\psi(x):=e^{-(l/2)\langle x\rangle^q}\in\SSS^{\{M_p\}}_{\{p!^{1/q}\}}(\RR^d)$ and pick $\theta\in\SSS^{\{p!^{1/q}\}}_{\{M_p\}}(\RR^d)\subseteq \SSS^{\{M_p\}}(\RR^d)$ such that $\langle g,\theta\rangle>0$. Choose $\chi\in\DD^{\{M_p\}}(\RR^d)$ such that $0\leq \chi\leq 1$, $\chi(x)=1$ if $|x|\leq 1$ and $\chi(x)=0$ when $|x|\geq 2$. For each $n\in\ZZ_+$, set $\psi_n(x)=\chi(x/2^n)\psi(x)\in\DD^{\{M_p\}}(\RR^d)$. Then
\begin{equation}\label{equ-for-cex-spa-addit}
\langle a,\psi_n\otimes \theta\rangle=\langle g,\theta\rangle\int_{\RR^d}e^{(l/2)\langle x\rangle^q} \chi(x/2^n)dx.
\end{equation}
Since $\psi_n\rightarrow \psi$ in $\SSS^{\{M_p\}}_{\{p!^{1/q}\}}(\RR^d)$, we infer that $\psi_n\otimes \theta\rightarrow \psi\otimes \theta$ in $\SSS^{\{M_p\}}_{\{p!^{1/q}\}}(\RR^d)\hat{\otimes} \SSS^{\{p!^{1/q}\}}_{\{M_p\}}(\RR^d)$. Consequently, the left-hand side of \eqref{equ-for-cex-spa-addit} tends to $\langle a,\psi\otimes \theta\rangle$. But, monotone convergence implies that the right-hand side of \eqref{equ-for-cex-spa-addit} tends to $\infty$.
\end{remark}

Finally, we consider the case when the functions $\varphi_1$ and $\varphi_2$ have very fast decay. We start with the following technical result.

\begin{lemma}\label{bou-brc-qsmlll}
The following inequality holds true
$$
\langle x-y\rangle^q\leq \langle x\rangle^q+\langle y\rangle^q,\quad \mbox{for all}\,\, x,y\in\RR^d,\, q\in[0,1].
$$
\end{lemma}

\begin{proof} When $q=0$, the inequality is trivial. For $0< q\leq1$, the claim follows from
$$
\langle x-y\rangle^{2q}\leq (1+|x|^2+|y|^2+2|x||y|)^q\leq \langle x\rangle^{2q}+\langle y\rangle^{2q}+2|x|^q|y|^q\leq (\langle x\rangle^q+\langle y\rangle^q)^2.
$$
\end{proof}

\begin{theorem}\label{ext221}
Let $\varphi_1(x)=\exp(-t_1 e^{\langle x\rangle^q})$ and $\varphi_2(x)=\exp(-t_2 e^{\langle x\rangle^q})$, $q,t_1,t_2>0$. Let $f_j\in L^1_{\operatorname{loc}}(\RR^d)$, $j=1,\ldots,k$, be such that $\exp(-e^{l\langle\cdot\rangle^q})f_j\in L^{\infty}(\RR^d)$, $j=1,\ldots,k$, for some $0<l<1$. Then for any $g_j\in\SSS'^{\{M_p\}}(\RR^d)$, $j=1,\ldots,k$, and ultradifferential operators $P_j$, $j=1,\ldots,k$, of class $\{M_p\}$, the ultradistribution
$$
a(x,\xi)= \sum_{j=1}^k(P_j(D_x)f_j(x))\otimes g_j(\xi)
$$
belongs to $\mathcal{A}_{\varphi_1,\varphi_2}$.
\end{theorem}

\begin{proof} Clearly, it suffices to prove the claim for $k=1$, i.e. when $a(x,\xi)=(P(D_x)f(x))\otimes g(\xi)$. Set $t:=\min\{t_1,t_2\}>0$. By \cite[Proposition 2.5]{PPV-JMPA}, there is an ultradifferential operator $\tilde{P}$ of class $\{M_p\}$ and a continuous function $\tilde{g}$ on $\RR^d$ with ultrapolynomial growth of class $\{M_p\}$ such that $g=\tilde{P}(D)\tilde{g}$. There is $(r_p)\in\mathfrak{R}$ such that $e^{-N_{r_p}(|\cdot|)}\tilde{g}\in L^{\infty}(\RR^d)$ (cf. \eqref{equ-for-con-equ-for-gro}). Since $\{M_p\}$ satisfies $(M.3)$, \cite[Corollary 1.4 $(a)$]{Petzsche88} implies that there exists $\rho>0$ such that $p!^{1+\rho}\subset M_p$. Fix $\tau\in(0,t)$ and set $\tilde{q}=\max\{q,1\}$. Proposition \ref{121} together with Proposition \ref{ex111} and Corollary \ref{iter2} give the following bounds: there are $h,c>0$ such that
\begin{gather}
\sup_{\alpha,\beta\in\NN^d}\sup_{x,\xi\in\RR^d}\frac{h^{|\alpha+\beta|} \exp(\tau e^{\langle x\rangle^{q}})e^{M(c|\xi|)} |\partial^{\alpha}_x\partial^{\beta}_{\xi}W(\varphi_2,\varphi_1)(x,\xi)|} {M_{\alpha}\beta!^{1/\tilde{q}}}<\infty;\label{bou-forest-lt1}\\
\sup_{\alpha,\beta\in\NN^d}\sup_{x,\xi\in\RR^d}\frac{h^{|\alpha+\beta|} \exp(\tau e^{\langle x\rangle^{q}})e^{M(c|\xi|)} |\partial^{\alpha}_x\partial^{\beta}_{\xi}P(D_x)\tilde{P}(D_{\xi}) W(\varphi_2,\varphi_1)(x,\xi)|} {M_{\alpha}\beta!^{1/\tilde{q}}}<\infty.\label{bou-for-ittr2}
\end{gather}
We claim that for any $\chi\in\DD^{\{M_p\}}(\RR^{2d})$, \eqref{tos-for-exi-con-dd} holds true; this will verify that the $\DD'^{\{M_p\}}$-convolution of $a$ and $W(\varphi_2,\varphi_1)$ exists. Let $\chi\in\DD^{\{M_p\}}(\RR^{2d})$ and set $\chi_1(x,\xi):= P(-D_x)\tilde{P}(-D_{\xi})\chi(x,\xi)$. Assume first $q\geq 1$. Employing \eqref{est-otr-papdjst} and \eqref{ine-for-seq-assoc}, we infer
\begin{align*}
|\check{a}*\chi(x,\xi)|&\leq C_1\int_{\RR^{2d}}\exp(e^{l\langle y-x\rangle^q}) e^{N_{r_p}(|\eta-\xi|)}|\chi_1(y,\eta)|dyd\eta\\
&\leq C_2 e^{N_{r_p}(2|\xi|)}\int_{\supp\chi_1}\exp(e^{l\langle x\rangle^q} e^{q2^{q-1}l|y|\langle x\rangle^{q-1}}e^{q2^{q-1}l\langle y\rangle^q}) e^{N_{r_p}(2|\eta|)}dyd\eta\\
&\leq C_3 e^{N_{r_p}(2|\xi|)}\exp(C'_2e^{l\langle x\rangle^q} e^{C'_1\langle x\rangle^{q-1}}),
\end{align*}
for some $C'_1,C'_2,C_3>0$ depending on $\supp\chi_1$. Now, the validity of \eqref{tos-for-exi-con-dd} follows from \eqref{bou-forest-lt1}. When $q\in(0,1)$, one employs Lemma \ref{bou-brc-qsmlll} instead of \eqref{est-otr-papdjst} to obtain $|\check{a}*\chi(x,\xi)|\leq C_3 e^{N_{r_p}(2|\xi|)}\exp(C'_2e^{l\langle x\rangle^q})$ with possibly different $C_3, C'_2>0$ and, again, \eqref{tos-for-exi-con-dd} follows from \eqref{bou-forest-lt1}. We conclude that the $\DD'^{\{M_p\}}$-convolution of $a$ and $W(\varphi_2,\varphi_1)$ exists.\\
\indent We claim that the function
$$
F:\RR^{2d}\rightarrow \CC,\, F(x,\xi)=\int_{\RR^{2d}}f(y)\tilde{g}(\eta) P(D_x)\tilde{P}(D_{\xi})W(\varphi_2,\varphi_1)(x-y,\xi-\eta)dyd\eta,
$$
is well-defined and $F\in\mathcal{C}^{\infty}(\RR^{2d})$. Pick $(r'_p)\in\mathfrak{R}$ such that
$$
(r'_p)\leq (r_p)\quad \mbox{and}\quad \sup_{\lambda>0}\lambda^{d+1}e^{N_{r_p}(\lambda)}e^{-N_{r'_p}(\lambda)}<\infty;
$$
for example, $r'_j:=r_1/H$, $j=1,\ldots,d+1$, and $r'_j:=r_{j-d-1}/H$, $j\geq d+2$. Assume first $q\geq 1$. Choose $l'\in(l,1)$. Let $K$ be an arbitrary but fixed compact subset of $\RR^d$. Employing \eqref{est-otr-papdjst}, \eqref{ine-for-seq-assoc} and \eqref{bou-for-ittr2}, for $\alpha,\beta\in\NN^d$ and $x\in K$, $\xi,\eta\in\RR^d$, $y\in\RR^d\backslash (\mbox{nullset})$, we infer
\begin{align*}
|f(y)|&|\tilde{g}(\eta)| |D^{\alpha}_xD^{\beta}_{\xi}P(D_x)\tilde{P}(D_{\xi})W(\varphi_2,\varphi_1)(x-y,\xi-\eta)| \langle y\rangle^{d+1}\langle \eta\rangle^{d+1}\\
&\leq C_4h^{-|\alpha|-|\beta|} M_{\alpha}\beta!^{1/\tilde{q}}\exp(e^{l'\langle y\rangle^q})e^{N_{r'_p}(|\eta|)} \exp(-\tau e^{\langle x-y\rangle^q}) e^{-M(c|\xi-\eta|)}\\
&\leq 2C_4h^{-|\alpha|-|\beta|} M_{\alpha}\beta!^{1/\tilde{q}} \exp(e^{l'\langle y-x\rangle^q} e^{q2^{q-1}l'|x|\langle y-x\rangle^{q-1}} e^{q2^{q-1}l'\langle x\rangle^q})\exp(-\tau e^{\langle x-y\rangle^q})\\
&{}\quad\cdot e^{N_{r'_p}(2|\eta-\xi|)} e^{N_{r'_p}(2|\xi|)} e^{-M(c|\xi-\eta|)}\\
&\leq C_5h^{-|\alpha|-|\beta|} M_{\alpha}\beta!^{1/\tilde{q}} e^{N_{r'_p}(2|\xi|)} \exp(C''_2e^{l'\langle x-y\rangle^q} e^{C''_1\langle x-y\rangle^{q-1}}) \exp(-\tau e^{\langle x-y\rangle^q}),
\end{align*}
with some $C''_1,C''_2>0$ which depend on $K$. Since $\exp(C''_2e^{l'\langle \cdot\rangle^q} e^{C''_1\langle \cdot\rangle^{q-1}}) \exp(-\tau e^{\langle \cdot\rangle^q})\in L^{\infty}(\RR^d)$, we deduce
$$
|f(y)||\tilde{g}(\eta)| |D^{\alpha}_xD^{\beta}_{\xi}P(D_x)\tilde{P}(D_{\xi}) W(\varphi_2,\varphi_1)(x-y,\xi-\eta)|\leq  \frac{C_6 M_{\alpha}\beta!^{1/\tilde{q}} e^{N_{r'_p}(2|\xi|)}}{h^{|\alpha|+|\beta|}\langle y\rangle^{d+1}\langle \eta\rangle^{d+1}},
$$
for all $x\in K$, $\xi,\eta\in\RR^d$, $y\in\RR^d\backslash(\mbox{nullset})$, $\alpha,\beta\in\NN^d$. Consequently, $F\in \mathcal{C}^{\infty}(\RR^{2d})$ and it satisfies the following bounds: for every compact subset $K$ of $\RR^d$ there is $C>0$ such that
\begin{equation}\label{est-for-add-fun-for-sym}
\frac{h^{|\alpha|+|\beta|}}{M_{\alpha}\beta!^{1/\tilde{q}}} |D^{\alpha}_xD^{\beta}_{\xi}F(x,\xi)|\leq Ce^{N_{r'_p/2}(|\xi|)},\quad \mbox{for all}\,\, x\in K,\, \xi\in\RR^d,\, \alpha,\beta\in\NN^d.
\end{equation}
When $q\in(0,1)$, one employs analogous technique but uses Lemma \ref{bou-brc-qsmlll} instead of \eqref{est-otr-papdjst} to show that $F\in\mathcal{C}^{\infty}(\RR^{2d})$ and that the bounds \eqref{est-for-add-fun-for-sym} hold true in this case as well. Now, by direct inspection, one verifies that $a*W(\varphi_2,\varphi_1)=F$ (cf. Remark \ref{rem-for-exi-con-alt-def-cor}) and \eqref{est-for-add-fun-for-sym} (specialised with $\beta=0$) proves the validity of \eqref{gelfand1}.
\end{proof}

\end{document}